%% file: paper.tex
\documentclass[twoside]{article}

\usepackage[accepted]{aistats2025}
\usepackage[round]{natbib}

\input{macros.tex}

\begin{document}

%

%

\twocolumn[

\aistatstitle{Differentially Private Random Block Coordinate Descent}

\aistatsauthor{ Artavazd Maranjyan \And Abdurakhmon Sadiev \And  Peter Richt\'{a}rik }

\aistatsaddress{ KAUST \And  KAUST \And KAUST } ]

\begin{abstract}

    Coordinate Descent (\algname{CD}) methods have gained significant attention in machine learning due to their effectiveness in solving high-dimensional problems and their ability to decompose complex optimization tasks. 
    However, classical \algname{CD} methods were neither designed nor analyzed with data privacy in mind, a critical concern when handling sensitive information. 
    This has led to the development of differentially private \algname{CD} methods, such as \algname{DP-CD} (Differentially Private Coordinate Descent) proposed by \cite{mangold_dp-cd}, yet a disparity remains between non-private \algname{CD} and \algname{DP-CD} methods. 
    In our work, we propose a differentially private random block coordinate descent method that selects multiple coordinates with varying probabilities in each iteration using sketch matrices. 
    Our algorithm generalizes both \algname{DP-CD} and the classical \algname{DP-SGD} (Differentially Private Stochastic Gradient Descent), while preserving the same utility guarantees. 
    Furthermore, we demonstrate that better utility can be achieved through importance sampling, as our method takes advantage of the heterogeneity in coordinate-wise smoothness constants, leading to improved convergence rates.
\end{abstract}

\section{Introduction}

Recently, there has been a growing interest in Coordinate Descent (\algname{CD}) methods due to their wide range of applications in machine learning, particularly for large or high-dimensional datasets. 
These methods effectively break down complex optimization problems into simpler subproblems, which can be easily parallelized or distributed \citep{wright2015coordinate, shi2016primer}. 
As a result, they are considered state-of-the-art for various optimization tasks \citep{nesterov2012efficiency, shalev2013stochastic, shalev2013accelerated, lin2014accelerated, fercoq2015accelerated, richtarik2016parallel, nesterov2017efficiency}.

Despite their effectiveness, most studies on coordinate descent methods do not consider privacy concerns. 
In machine learning, handling sensitive or confidential data poses significant challenges, as emphasized by \cite{shokri2017membership}. 
The potential risks associated with data leakage necessitate robust mechanisms to ensure that individual data points cannot be easily reconstructed or identified from the trained models. 
Differential privacy (DP) has emerged as a key strategy for mitigating these issues. 
A common approach to ensure privacy while training models is to formulate the problem as an empirical risk minimization (ERM) task under DP constraints, as outlined by \citet{chaudhuri2011differentially}.

Recently, \cite{mangold_dp-cd} introduced a differentially private version of coordinate descent, referred to as \algname{DP-CD}. 
While this method has made strides in addressing privacy concerns, a gap remains between traditional coordinate descent methods and \algname{DP-CD}. 
In particular, \algname{DP-CD} selects only one coordinate per iteration, which can limit its efficiency in high-dimensional settings. 

In this work, we aim to bridge this gap by allowing for the selection of multiple coordinates with varying probabilities in each iteration. 
This approach not only enhances flexibility in optimization but also improves convergence rates, especially in scenarios where certain coordinates may have more significant impacts on the objective function than others. 
By leveraging block coordinate descent strategies, our method can better capture the underlying structure of the problem, leading to more efficient solutions while maintaining the privacy guarantees required in sensitive applications.

\section{Problem formulation and Contributions}

Our work focuses on designing a differentially private algorithm to approximate the solution to an empirical risk minimization (ERM) problem, defined as:
$$
  w^{\star} \in \argmin_{w \in \mathbb{R}^d} \left\{ f(w) \eqdef \frac{1}{n} \sum_{i=1}^{n} \ell(w; \zeta_i) \right\},
$$
where $\ell(w; \zeta_i): \R^d \times \cX \rightarrow \R$ is the loss function for a sample $\zeta_i$, and $D = (\zeta_1, \dots, \zeta_n)$ is a dataset of $n$ samples drawn from the universe $\cX$.

In \citet{mangold_dp-cd}, the authors tackle the problem of minimizing a composite ERM objective, which includes a nonsmooth convex regularizer:
$$
  \min_{w \in \mathbb{R}^d} \left\{ f(w) \eqdef \frac{1}{n} \sum_{i=1}^{n} \ell(w; \zeta_i) + \psi(w) \right\}.
$$
However, they assume the regularizer $\psi$ is separable, meaning $\psi(w) = \sum_{i=1}^{d} \psi_i(w_i)$. 
We argue that this assumption is unnecessary, as it is not required in the optimization literature (see, e.g., \citet{hanzely2020variance, safaryan2021smoothness}). 

Addressing the case where $\psi$ is non-separable involves more advanced techniques, which we leave as future work. 
Further discussion on this can be found in \Cref{section_conclusion_future_work}.

\subsection{Summary of Contributions}

In this work, we present several key contributions that advance the understanding and application of differentially private coordinate descent methods:

\begin{itemize}
  \item In \Cref{section:algorithm}, we introduce the \algname{DP-SkGD} algorithm, as detailed in \Cref{algorithm}. 
  This method is a differentially private adaptation of the \algname{SkGD} approach proposed by \citet{safaryan2021smoothness}. 
  Our algorithm leverages sketches to efficiently select multiple coordinates in each iteration, which enhances its performance and flexibility. 
  As a result, \algname{DP-SkGD} generalizes the \algname{DP-CD} algorithm by \citet{mangold_dp-cd} by allowing for this multi-coordinate selection. 
  Additionally, it extends the original \algname{SkGD} method not only by incorporating differential privacy and noise but also by using diagonal matrix step sizes. 
  Furthermore, our analysis encompasses the convex regime, providing a broader theoretical framework for understanding the algorithm's behavior.

  \item In \Cref{section:block_sampling}, we investigate various sampling strategies for generating sketches, deriving utility bounds for each approach. 
  These strategies enable our algorithm to select multiple coordinates effectively while maintaining differential privacy. 
  The results are summarized in \Cref{table}, allowing for easy comparison of the effectiveness of different sampling methods.
  
  \item In \Cref{section:comparison}, we conduct a thorough comparison of our algorithm with \algname{DP-SGD}, \algname{DP-CD}, and \algname{DP-SVRG}. 
  We highlight scenarios where each method outperforms the others, particularly emphasizing that our method can achieve a speedup of up to $\sqrt{d}$ in the case of full sampling compared to \algname{DP-SGD}. 
  Furthermore, when employing importance sampling, our algorithm demonstrates a potential speedup of up to $\sqrt{d}$ over \algname{DP-CD}.
  
  \item In \Cref{section:experiments}, we evaluate the practical performance of our algorithm through extensive experiments, validating its effectiveness and efficiency in real-world applications.
\end{itemize}

\section{Preliminaries and Assumptions} \label{section_preliminaries}

In this section, we introduce key technical concepts that will be utilized throughout the paper.
Throughout the paper, we will frequently use the notation $[d]\eqdef \{1,\ldots, d\}$.

\subsection{Norms}
We start by defining two conjugate norms, which are crucial for our analysis as they help in tracking coordinate-wise quantities. 
Let $\langle u,v \rangle = \sum_{j=1}^d u_j v_j$ represent the Euclidean inner product, and define $\mM = \Diag{M_1, \dots, M_d}$, where $M_1, \ldots, M_d > 0$. 
We then introduce the following norm:
\begin{align*}
  \norm{w}_{\mM} &= \sqrt{\langle \mM w, w \rangle}.
\end{align*}
When $\mM$ is the identity matrix $\mI$, the $\mI$-norm $\norm{\cdot}_{\mI}$ corresponds to the standard $\ell_2$-norm $\norm{\cdot}_2$.

\subsection{Random Sets}
Next, to address random block coordinate descent, we need to introduce the notion of random sets in $[d]$.
Let $\cS$ be a probability distribution over the $2^d$ subsets of the coordinates/features of the model $x \in \R^d$ that we wish to train. 
Given a random set $S \sim \cS$, define
$$
p_j \eqdef \Pr(j \in S), \quad j \in [d].
$$
We also denote $\mP = \Diag{p_1, \ldots, p_d}$.
For this work, we restrict our attention to proper and nonvacuous random sets.
\begin{assumption}
  \label{ass:nonvacuous_proper}
  Let $S \sim \cS$ be {\em nonvacuous}, i.e., $P(S = \emptyset) = 0$, and {\em proper}, meaning that $p_j > 0$ for all $j \in [d]$.
\end{assumption}

\subsection{Sketch and Sparsification}

We study unbiased diagonal sketches, defined as follows:

\begin{definition}[Unbiased diagonal sketch]\label{def_sketch}
For a given random set $S\sim \cS$ we define a random diagonal matrix (sketch) $\mC = \mC(S) \in \R^{d\times d}$  via 
\begin{equation}\label{sketch-matrix-C}
  \mC = \Diag{\mathrm{c}_1, \dots, \mathrm{c}_d}, \quad \mathrm{c}_j = \begin{cases}
     \frac{1}{p_j}, & \;\text{if}\;j\in S,\\
     0,  & \text{otherwise}. 
    \end{cases}
\end{equation}
Equivalently, we can write
$$
\mC = \mI_{S} \mP^{-1},
$$
where $\mI_S = \Diag{\delta_1, \delta_2, \dots, \delta_d}$ is a diagonal matrix with
$$
\delta_i = 
\begin{cases}
1, & \text{if} \quad i \in S, \\
0, & \text{if} \quad i \notin S.
\end{cases}
$$
\end{definition} 
Note that given a vector $x = (x_1,\dots, x_d)\in \R^d$, we have
$$
(\mC x)_j = \begin{cases}
   \frac{x_j}{p_j}, & \text{if} \qquad j\in S, \\
  0, & \text{if} \qquad j\notin S.
\end{cases}
$$
Thus, we can control the sparsity level of the product $\mC x$ by engineering the properties of the random set $S$. 
Also note that $\E{\mC x} = x$ for all $x$.

\subsection{Assumptions}
We review the classical regularity assumptions, along with those specific to the coordinate-wise setting.

\begin{assumption}[Differentiability]
  \label{ass:differentiability}
  Function $\ell(\cdot;\zeta) : \R^d \times \cX \rightarrow \R$ is differentiable for all $\zeta \in \cX$.
\end{assumption}

\begin{assumption}[Convexity]
  \label{ass:convexity}
  Function $\ell(\cdot;\zeta) : \R^d \times \cX \rightarrow \R$ is convex for all $\zeta \in \cX$. 
  That is, for all $v, w \in \R^d$, 
  $$
  \ell(w;\zeta) \ge \ell(v;\zeta) + \langle \nabla \ell(v;\zeta), w - v \rangle.
  $$
\end{assumption}

Since we also study our algorithm in the strongly convex regime we will also need the following assumption.
\begin{assumption}[Strong convexity]
  \label{ass:strong_convexity}
  Function $f : \R^d \rightarrow \R$ is $\mu_{\mM}$-strongly-convex \wrt the norm $\smash{\norm{\cdot}_{\mM}}$.
  That is, for all $v, w \in \R^d$, 
  $$
  f(w) \ge f(v) + \scalar{\nabla f(v)}{w - v} + \frac{\mu_{\mM}}{2}\norms{w - v}_{\mM}.
  $$
\end{assumption}
The case $M_1=\cdots=M_d=1$ recovers standard $\mu_{\mI}$-strong convexity \wrt the $\ell_2$-norm.
Note that we do not assume strong convexity for $\ell$.

\begin{assumption}[Component smoothness]
  \label{ass:comp_smoothness}
  Function $f : \R^d \rightarrow \R$ is $\mM$-component-smooth for $M_1,\dots,M_d > 0$.
  That is, for all $v, w \in \R^d$, 
  $$
  f(w) \le f(v) + \scalar{\nabla f(v)}{w - v} + \frac{1}{2}\norms{w - v}_{\mM}.
  $$
\end{assumption}  
When $M_1=\dots=M_d=\beta$, $f$ is said to be $\beta$-smooth.

\begin{assumption}[Component Lipschitzness]
  \label{ass:comp_lipschitz}
  Let $\cS$ be a probability distribution over the $2^d$ subsets of $[d]$. 
  Function $\ell(\cdot;\zeta) : \R^d \times \cX \rightarrow \R$ is $L_\cS$-component-Lipschitz with $L_U > 0$ for all $U \in \range{\cS}$, for all $\zeta \in \cX$ \footnotemark. 
  \footnotetext{A better definition could involve using set covers (coverings), but we avoid this approach to prevent introducing additional notation.}
  This means that for all $v, w \in \R^d$, we have:
  $$
  \abs{\ell\(w + \mI_{U}v; \zeta\) - \ell\(w; \zeta\)} \leq L_U \norm{\mI_{U} v}.
  $$
\end{assumption} 
This is a more refined assumption compared to the classical Lipschitzness assumption, which it also generalizes. 
Moreover, we have the following special cases depending on the distribution $\cS$. Given $S \sim \cS$:
\begin{itemize}
  \item \textbf{Full sampling.} 
  If $S = [d]$ with probability one, then we recover the classical Lipschitzness assumption, which can be equivalently written as
  $$
  \abs{\ell(v; \zeta) - \ell(w; \zeta)} \leq L \norm{v - w}_2,
  $$
  for all $v, w \in \mathbb{R}^d$.
  In this case, we say that $\ell$ is $L$-Lipschitz.
  \item \textbf{Single coordinate sampling.} 
  If $S$ is singleton, i.e., $S = \{j\}$ with some probability $p_j > 0$, we obtain coordinate-wise Lipschitzness with $L_{\cS} = (L_1, \ldots, L_d)$, where $L_j > 0$ for all $j \in [d]$. 
  Specifically, we have:
  $$
  \abs{\ell(w + t e_j; \zeta) - \ell(w; \zeta)} \leq L_j \abs{t},
  $$
  for all $w \in \mathbb{R}^d$ and $t \in \mathbb{R}$, where $e_j$ is the $j$-th vector in the canonical basis of $\mathbb{R}^d$. 
  This assumption is stated in \citet{mangold_dp-cd}.
  \item \textbf{Block sampling.}  
  Consider a partition of $[d]$ into $b$ nonempty blocks, denoted as $A_1, \dots, A_b$.  
  Let $S = A_j$ with probability $q_j > 0$, where $\sum_j q_j = 1$.  
  Under this sampling scheme, we obtain block-wise Lipschitzness, characterized by $L_{\cS} = (L_1, \ldots, L_b)$, with $L_j > 0$ for all $j \in [b]$. 
  Specifically, we have:
  $$
  \abs{\ell\(w + \mI_{C_i}v; \zeta\) - \ell\(w; \zeta\)} \leq L_i \norm{\mI_{C_i} v}.
  $$  
\end{itemize}

Note that block sampling generalizes both single-coordinate sampling and full sampling. 
However, it is not the most general case. 
For instance, consider the case where $S$ is a random subset of $[d]$ of size $\tau$, chosen uniformly at random. 
This is known as \textbf{nice sampling}, and under this sampling, our assumption still holds.

\subsection{Differential Privacy (DP)}
Next, we recall the concept of differential privacy.
Let $\cD$ be a set of datasets and $\cV$ a set of possible outcomes.
Two datasets $D, D' \in \cD$ are said \textit{neighboring} (denoted by $D \sim D'$) if they differ on at most one element.

\begin{definition}[Differential Privacy, \citep{dwork2008differential}]
  A randomized algorithm
  $\cA : \mathcal D
    \rightarrow \cV$ is $(\epsilon, \delta)$-differentially private if,
  for all neighboring datasets $D, D' \in \mathcal D$ and all
  $U \subseteq \cV$ in the range of $\cA$:
  \begin{align*}
    \Prob{\cA(D) \in U} \le \exp(\epsilon) \Prob{\cA(D') \in U} + \delta .
  \end{align*}
\end{definition}
The value of a function $h: \mathcal D \rightarrow \mathbb R^p$ can be privately released using the Gaussian mechanism, which adds centered Gaussian noise to $h(D)$ before releasing it \citep{dwork2014algorithmic}.
The scale of the noise is calibrated to the sensitivity $\Delta(h) = \sup_{D \sim D'} \norm{h(D) - h(D')}_2$ of $h$.

In our context, we are interested in the sensitivity arising from specific coordinates. 
Let $U\subseteq[d]$ represent a set of coordinates. 
We define the sensitivity as follows:
$$
\Delta_U(h) \eqdef \sup_{D \sim D'} \norm{\mI_U \(h(D) - h(D')\)}_2.
$$
We are interested in analyzing the sensitivity of $\ell$.
Given that $\ell$ satisfies the component-Lipschitzness property (\Cref{ass:comp_lipschitz}), we can bound its sensitivity as follows.
\begin{lemma}[Proof in \Cref{section:proof_sensitivity}]
  \label{lem:sensitivity}
  Let $\ell: \R^d \times \cX \rightarrow \R$ be differentiable in its first argument (\Cref{ass:differentiability}), convex (\Cref{ass:convexity}), and $L_{\cS}$-component-Lipschitz with $L_U > 0$ for all $U \in \range{\cS}$ (\Cref{ass:comp_lipschitz}).
  Then 
  $$
  \Delta_U\(\nabla \ell \) \leq 2 L_U
  $$
  for all $U \in \range{\cS}$.
\end{lemma}

In this paper, we focus on the classical central model of differential privacy (DP), where a trusted curator has access to the raw dataset and releases a trained model based on that data.

\section{Differentially Private Sketched Gradient Descent}
\label{section:algorithm}

We propose the \algname{DP-SkGD} algorithm, as shown in \Cref{algorithm}. 
This method is a differentially private (DP) adaptation of the \algname{SkGD} algorithm introduced by \citet{safaryan2021smoothness}. 
The key innovation in \algname{DP-SkGD} is the integration of differential privacy into the sketched gradient framework, which enables efficient updates while preserving privacy.

The algorithm uses sketch matrices $\mathbf{C}$, as defined in \eqref{sketch-matrix-C}, to update on a randomly selected subset of coordinates. 
At each iteration, based on a randomly drawn subset $S$, a noisy gradient update is performed, adjusting only the coordinates in $S$. 
This process makes the algorithm particularly well-suited for large-scale optimization tasks where updating all coordinates in each step is computationally expensive.

The variance of the added noise is proportional to the coordinates selected in $S$, allowing for different noise scales $\sigma_S$. 
This means that coordinates with larger Lipschitz constants or those that are more sensitive can receive more noise to preserve differential privacy, while less sensitive coordinates can be updated with less noise. 
This structure allows for a more fine-tuned noise injection compared to global, uniform noise addition seen in some other DP algorithms.

Furthermore, by setting different step sizes $\gamma_i$ for each coordinate (encoded in the diagonal matrix $\mGamma$), the updates are directionally unbiased, meaning that each coordinate is updated in a way that reflects its specific geometry and properties, thus preserving the integrity of the gradient direction. 
This is a common property of \algname{SGD}-type methods, ensuring that the overall method converges to the correct solution even when updates are performed on different subsets of coordinates at different iterations.

\begin{algorithm}[H]
    \begin{algorithmic}[1]
    \STATE \textbf{Input:} Initial point $w^0\in\R^d$, step sizes $\mGamma = \Diag{\gamma_1,\ldots, \gamma_d}$, number of iterations $T$, number of inner loops $K$, probability distribution $\cS$ over the subsets of $[d]$, noise scales $\sigma_U$ for $U\in \range{\cS}$.
    \FOR{$t = 0,\ldots, T-1$}
    \STATE Set $\theta^0 = w^{t}$
    \FOR{$k = 0,\ldots, K-1$}
    \STATE Sample a subset $S \sim \cS$ and let $\mC = \mC(S)$\\ (see definition \ref{def_sketch})
    \STATE Draw $\eta \sim \cN\(0, \sigma_{S} \mI\)$
    \STATE $\theta^{k+1} = \theta^k - \mGamma\mC\(\nabla f(\theta^k) + \eta \)$
    \ENDFOR
    \STATE $w^{t+1} = \frac1K \sum_{k=1}^K \theta^k$
    \ENDFOR 
    \end{algorithmic}
    \caption{\algname{DP-SkGD}} \label{algorithm}
  \end{algorithm}

In the special case where the distribution $\mathcal{S}$ is uniform over single elements in $[d]$, the proposed algorithm simplifies to \algname{DP-CD}, the \algname{Differentially Private Coordinate Descent} method introduced by \citet{mangold_dp-cd}. 
Additionally, if $S=[d]$ with probability one it reduces to \algname{DP-SGD} \citep{bassily2014Private}.

\subsection{Privacy Guarantees}
For \Cref{algorithm} to satisfy $(\epsilon,\delta)$-DP, the noise scales $\sigma_S$ should be calibrated as given by the theorem bellow.

\begin{theorem}[Proof in \Cref{proof_privacy}] 
  \label{thm_privacy}
  Let $\cS$ be nonvacuous and proper probability distribution over the subsets of $[d]$ (\Cref{ass:nonvacuous_proper}).
  Let $\ell: \R^d \times \cX \rightarrow \R$ be differentiable in its first argument (\Cref{ass:differentiability}), convex (\Cref{ass:convexity}), and $L_{\cS}$-component-Lipschitz with $L_U > 0$ for all $U \in \range{\cS}$ (\Cref{ass:comp_lipschitz}).
  Let $0 < \epsilon \leq 1$, $\delta < 1/3$. 
  If 
  $$
  \sigma_U^2 = \frac{12 L_U^2 K T \log(1/\delta)}{n^2\epsilon^2},
  $$
  for all $U \in \range{\cS}$, then \Cref{algorithm} satisfies $(\epsilon, \delta)$-DP.
\end{theorem}

\subsection{Utility Guarantees}

Here we state the utility results of \algname{DP-SkGD}.

\begin{theorem}[Proof in \Cref{proof_utility}] \label{thm_utility}
  Let $\cS$ be a nonvacuous and proper probability distribution over the subsets of $[d]$ (\Cref{ass:nonvacuous_proper}).
  Define $\mP = \Diag{p_1, \dots, p_d}$, where $p_i = \Pr(i \in S)$ denotes the probability that the $i$-th component is included in the random set $S \sim \cS$.
  Let $\ell: \R^d \times \cX \rightarrow \R$ be differentiable in its first argument (\Cref{ass:differentiability}), convex (\Cref{ass:convexity}), and $L_{\cS}$-component-Lipschitz with $L_U > 0$ for all $U \in \range{\cS}$ (\Cref{ass:comp_lipschitz}).
  Further, let $f$ be $\mM$-component-smooth (\Cref{ass:comp_smoothness}), and let $w^{\star}$ denote a minimizer of $f$, with $f^{\star} = f(w^{\star})$ representing the minimum value.
  Consider $w_{\text{priv}} \in \R^d$ as the output of \Cref{algorithm} with step sizes $\mGamma = \mP\mM^{-1}$.
  We assume a privacy budget with $\epsilon \leq 1$ and $\delta < 1/3$.
  The noise scales $\sigma_U$ are set as in \Cref{thm_privacy} to ensure $(\epsilon,\delta)$-DP, with $T$ and $K$ chosen below.
  Denote 
  $$
  \Sigma_S^2 = \E{\|\mC L_S \one \|_{\mP \mM^{-1}}^2},
  $$
  where $\one$ is a vector of all ones, i.e., $\one = (1, \ldots, 1)^\top$.
  \begin{itemize}
    \item \textbf{Convex Case.} 
      Set
      $$
      T = 1 \quad\text{and}\quad K = \frac{n\varepsilon R_{\mM\mP^{-1}}}{\Sigma_S \sqrt{\log(1/\delta)}}.
      $$
      Then, the utility of \Cref{algorithm} is given by 
        \begin{equation*}
            \E{f(w^{\text{priv}}) - f^{\star}}  = \cO \(\frac{\Sigma_S R_{\mM\mP^{-1}}}{n\varepsilon}\log^{\nicefrac{1}{2}} \frac{1}{\delta}\),
        \end{equation*}
      where $R^2_{\mM\mP^{-1}} = \norms{w^0 - w^{\star}}_{\mM\mP^{-1}}$.
      
    \item \textbf{Strongly Convex Case.} 
      Assume $f$ is $\mu$-strongly convex (\Cref{ass:strong_convexity}).
      Set 
      $$
      K = 2\(1+\frac{1}{{\mu}\max_{i\in[d]}\left\{\frac{M_i}{p_i}\right\}}\),
      $$
      $$
      T = \log_2\left(\frac{(f(w^0) - f^{\star}) n^2 \epsilon^2}{\(1+\frac{1}{{\mu}}\max_{i\in[d]}\left\{\frac{M_i}{p_i}\right\}\) \Sigma_S^2 \log(1/\delta)}\right).
      $$
      Then, the utility of \Cref{algorithm} is given by 
      \begin{equation*}
          \E{f(w^T) - f^{\star}} = \widetilde \cO\left(\frac{1}{\mu}\max_{i\in[d]}\left\{\frac{M_i}{p_i}\right\} \Sigma_S^2 \frac{\log(1/\delta)}{n^2 \epsilon^2}\right).
      \end{equation*}
  \end{itemize}
\end{theorem}

We use the asymptotic notation $\widetilde \cO$ to suppress insignificant logarithmic factors. 
Detailed, non-asymptotic utility bounds are provided in \Cref{proof_utility}.

A crucial quantity in our results is
$$
\Sigma_S^2 = \E{\norms{\mC L_S \one}_{\mP \mM^{-1}}} = \E{L_S^2 \norms{\mI_S \one }_{\mP^{-1} \mM^{-1}}},
$$
which varies based on the sampling strategy for $S$.
Since block sampling generalizes both full sampling and single-coordinate sampling as special cases, we will focus on this generalized scenario.

\subsubsection{Block Sampling}
\label{section:block_sampling}
\begin{table*}[t]
  \centering
  \caption{
    Utility guarantees for \algname{DP-SkGD-BS} with varying values of $b$ and different sampling strategies, along with \algname{DP-CD}, \algname{DP-SGD}, and \algname{DP-SVRG}.
  } \label{table}
    \begin{tabular*}{\textwidth}{c @{\extracolsep{\fill}} c c c}
      \toprule
      & Convex
      & Strongly-convex \\
      \midrule
      \algname{DP-SkGD-BS} (this paper)
      & $\cO_* \(\norm{L_{\{A_1,\ldots,A_b\}}}_{\mM^{-1}} R_{\mM\mP^{-1}}\)$
      & $\widetilde \cO_{-} \left(\norms{L_{\{A_1,\ldots,A_b\}}}_{\mM^{-1}} \frac{1}{\mu}\max\limits_{i\in[d]}\left\{\frac{M_i}{p_i}\right\} \right)$\\
      Uniform Sampling
      & $\cO_* \left(\norm{L_{\{A_1,\ldots,A_b\}}}_{\mM^{-1}} R_{\mM} \sqrt{b} \right)$
      & $\widetilde \cO_{-} \left(\norms{L_{\{A_1,\ldots,A_b\}}}_{\mM^{-1}} \frac{1}{\mu}M_{\max} b \right)$\\
      Importance Sampling
      & $\cO_* \left( \norm{L_{\{A_1,\ldots,A_b\}}}_{\mM^{-1}} R_{\mI} \sqrt{\sum\limits_{i=1}^b \max\limits_{j\in A_i} M_j}\right)$
      & $\widetilde \cO_{-} \left( \norm{L_{\{A_1,\ldots,A_b\}}}_{\mM^{-1}} \frac{1}{\mu} \sum\limits_{i=1}^b \max\limits_{j\in A_i} M_j \right)$ \\
      \midrule
      \makecell{\algname{DP-SkGD-BS} (this paper) \\ $b=d$}
      & $\cO_* \left(\norm{L_{\{1,\ldots,d\}}}_{\mM^{-1}} R_{\mM\mP^{-1}} \right)$
      & $\widetilde \cO_{-} \left(\norms{L_{\{1,\ldots,d\}}}_{\mM^{-1}} \frac{1}{\mu}\max\limits_{i\in[d]}\left\{\frac{M_i}{p_i}\right\} \right)$\\
      Uniform Sampling
      & $\cO_* \left(\norm{L_{\{1,\ldots,d\}}}_{\mM^{-1}} R_{\mM} \sqrt{d} \right)$
      & $\widetilde \cO_{-} \left(\norms{L_{\{1,\ldots,d\}}}_{\mM^{-1}} \frac{1}{\mu}M_{\max} d  \right)$\\
      Importance Sampling
      & $\cO_* \left(\norm{L_{\{1,\ldots,d\}}}_{\mM^{-1}} R_{\mI} \sqrt{\tr{\mM}} \right)$
      & $\widetilde \cO_{-} \left(\norms{L_{\{1,\ldots,d\}}}_{\mM^{-1}} \frac{1}{\mu}\tr{\mM}  \right)$\\
      \midrule
      \algname{DP-CD} \citep{mangold_dp-cd}
      & $\cO_* \left(\norm{L_{\{1,\ldots,d\}}}_{\mM^{-1}} R_{\mM} \sqrt{d} \right)$
      & $\widetilde \cO_{-}\left(\norms{L_{\{1,\ldots,d\}}}_{M^{-1}} \frac{1}{\mu_{\mM}}d \right)$\\
      \midrule
      \makecell{\algname{DP-SkGD-BS} (this paper) \\ $b=1$}
      & $\cO_* \left(L\sqrt{\tr{\mM^{-1}}} R_{\mM} \right)$
      & $\widetilde \cO_{-} \left(L^2\tr{\mM^{-1}} \frac{1}{\mu} M_{\max}\right)$\\
      \midrule
      \makecell{\algname{DP-SGD} \citep{bassily2014Private} \\ \algname{DP-SVRG} \citep{wang2017Differentially}}
      & $\cO_* \left(L \sqrt{d} R_{\mI}\right)$
      & $\widetilde \cO_{-} \left(L^2 \frac{1}{\mu_{\mI}} d\right)$ \\
      \bottomrule
    \end{tabular*}
    We use the notation $\mathcal{O}_*$ to suppress the common term $\frac{\sqrt{\log(1/\delta)}}{n\epsilon}$, which appears consistently across all rates. 
    Similarly, we denote $\mathcal{O}_{-}$ to suppress the term $\frac{\log(1/\delta)}{n^2 \epsilon^2}$, as it is also consistent across all rates.\\
\end{table*}

Consider a partition of $[d]$ into $b$ nonempty blocks, denoted as $A_1, \dots, A_b$.
Let $S = A_j$ with probability $q_j > 0$, where $\sum_j q_j = 1$.  
For each $i \in [n]$, let $B(i)$ indicate which block $i$ belongs to. 
In other words, $i \in A_j$ iff $B(i) = j$. 
Then $p_i := \text{Prob}(i \in S) = q_{B(i)}$.
We call the resulting method \algname{DP-SkGD-BS}.
Then we have 
\begin{align*}
  \Sigma_S^2 
  &= \E{L_S^2 \norms{\mI_S\one}_{\mP^{-1} \mM^{-1}}} = \sum_{i=1}^b q_i L_{A_i}^2 \norms{\mI_{A_i} \one}_{\mP^{-1} \mM^{-1}}\\
  &= \sum_{i=1}^b q_i L_{A_i}^2 \sum_{j\in A_i} \frac{1}{q_i M_j} = \norms{L_{\{A_1,\ldots,A_b\}}}_{\mM^{-1}},
\end{align*}
where we define 
$$
L_{\{A_1,\ldots,A_b\}} := \sum_{i=1}^d L_{B(i)}e_i.
$$
For the convex case the rate for block sampling becomes
$$
\cO \(\frac{\norm{L_{\{A_1,\ldots,A_b\}}}_{\mM^{-1}} R_{\mM\mP^{-1}}}{n\varepsilon}\log^{\nicefrac{1}{2}} \frac{1}{\delta}\).
$$
Note that we still have flexibility in choosing the probabilities $q_i$. 
Ideally, the best choice would minimize the term
$$
R_{\mM\mP^{-1}} = \norms{w^0 - w^{\star}}_{\mM\mP^{-1}},
$$
but since $w^{\star}$ is unknown, this approach is not feasible. 
However, an alternative is to use importance sampling. 
Specifically, we can choose
$$
q_i = \frac{\max_{j\in A_i} M_j}{\sum_{i=1}^b \max_{j\in A_i} M_j},
$$
which leads to the bound
$$
\mM\mP^{-1} \preceq \mI \sum_{i=1}^b \max_{j\in A_i} M_j.
$$
This results in the following utility bound:
$$
\cO \left(\frac{\norm{L_{\{A_1,\ldots,A_b\}}}_{\mM^{-1}} R_{\mI} \sqrt{\sum_{i=1}^b \max_{j\in A_i} M_j}}{n\varepsilon} \log^{\nicefrac{1}{2}} \frac{1}{\delta}\right).
$$
\begin{itemize}
  \item \textbf{Single Coordinate Sampling:} 
  In this case, we have
  $$
  \Sigma_S^2 = \norms{L_{\{1,\ldots,d\}}}_{\mM^{-1}},
  $$
  which leads to the bound
  $$
  \cO \left( \frac{\norm{L_{\{1,\ldots,d\}}}_{\mM^{-1}} R_{\mM\mP^{-1}}}{n\varepsilon} \log^{\nicefrac{1}{2}} \frac{1}{\delta} \right).
  $$
  If we choose $p_i = \frac{1}{d}$, we obtain the simplified result:
  $$
  \cO \left( \frac{\norm{L_{\{1,\ldots,d\}}}_{\mM^{-1}} \sqrt{d} R_{\mM}}{n\varepsilon} \log^{\nicefrac{1}{2}} \frac{1}{\delta} \right).
  $$
  With importance sampling, the result becomes
  $$
  \cO \left( \frac{\norm{L_{\{1,\ldots,d\}}}_{\mM^{-1}} \sqrt{\tr{\mM}} R_{\mI}}{n\varepsilon} \log^{\nicefrac{1}{2}} \frac{1}{\delta} \right).
  $$
  
  \item \textbf{Full Sampling:} 
  When we sample the full set $S = [d]$ with probability 1, we have
  $$
  \Sigma_S^2 = \norms{L \one}_{\mM^{-1}} = L^2\tr{\mM^{-1}},
  $$
  leading to the bound
  $$
  \cO \left( \frac{L\sqrt{\tr{\mM^{-1}}} R_{\mM}}{n\varepsilon} \log^{\nicefrac{1}{2}} \frac{1}{\delta} \right).
  $$
\end{itemize}

These and other special cases are summarized in \Cref{table}.

\subsection{Comparison}
\label{section:comparison}
Let us now discuss the scenarios in which our method outperforms previous approaches.

\subsubsection{Comparing to DP-SGD}
We consider the convex case, where the comparison reduces to analyzing 
$$
\sqrt{\tr{\mM^{-1}}} R_{\mM} \quad \text{versus} \quad \sqrt{d} R_{\mI}.
$$
If $M = M_1 = \cdots = M_d$, these two quantities are identical.

Now, suppose we have $M_j := M_{\max} \gg M_{\min} := M_1$ for all $j \neq 1$, such that $\tr{\mM^{-1}} \approx \frac{1}{M_{\min}}$. In this case, the first coordinate is small while the others are significantly larger. Imagine that for the initial iterate $w^0$, the coordinates $j \neq 1$ are already very close to their optimal values, such that
$$
\sum_{j \neq 1} M_j | w_j^0 - w_j^* | \ll M_1 | w_1^0 - w_1^\star|.
$$
Under this assumption,
$$
R_{\mM}^2 \approx M_1 | w_1^0 - w_1^\star|^2 \approx M_{\min} R_{\mI}^2.
$$
Thus, we obtain
$$
\sqrt{\tr{\mM^{-1}}} R_{\mM} \approx R_{\mI},
$$
resulting in a rate that is approximately $\sqrt{d}$ times faster. 
This outcome is expected because we assume that the problem is nearly solved for all coordinates except one, and that solving this single coordinate is easy due to the small value of $M_1$. 
Essentially, we are running \algname{SGD} on just one coordinate, leading to the observed acceleration by a factor of $\sqrt{d}$.

A similar argument can be made for the strongly convex case.
We need to use that, if $f$ is $\mu$-strongly convex then it is $\frac{\mu}{M_{\max}}$-strongly strongly convex with respect to the norm $\|\cdot\|_{\mM}$.

\paragraph{Block Sampling.}
Assume that $M_{\max} := M_i \gg M_{\min} := M_j$ for all $i \in A_1$ and $j \notin A_1$. 
Additionally, let $L_{\max} := L_1 \gg L_i$ for all $i \neq 1$. If $L_1^2 / M_{\max}$ dominates the other terms, we have
$$
\norm{L_{\{A_1,\ldots,A_b\}}}_{\mM^{-1}} \approx \frac{L_{\max}}{M_{\max}} \approx \frac{L}{M_{\max}}.
$$
Furthermore, if $w^0$ is such that 
$$
\norm{L_{\{A_1,\ldots,A_b\}}}_{\mM^{-1}} R_{\mM} \approx \sqrt{\frac{M_{\min}}{M_{\max}}} L R_{\mI},
$$
then \algname{DP-SkGD-BS} with block sampling can significantly outperform \algname{DP-SGD} or \algname{DP-SVRG}, achieving arbitrarily better results in certain cases.

\subsubsection{Comparison with DP-CD}

\paragraph{Importance Sampling.}
Our method gains an advantage over \algname{DP-CD} due to the use of importance sampling.

To illustrate, consider the case where $b = d$ (i.e., single coordinate sampling). 
Assume that $M_1 \gg M_j$ for all $j \neq 1$, and similarly, $| w_1^0 - w_1^\star | \gg | w_j^0 - w_j^\star |$ for all $j \neq 1$. Moreover, suppose $M_1 | w_1^0 - w_1^\star | \gg M_j| w_j^0 - w_j^\star |$. 
Then, in the convex case, we get
$$
R_{\mI} \sqrt{\tr{\mM}} \approx \sqrt{M_1 | w_1^0 - w_1^\star |} \approx R_{\mM}.
$$
Thus, \algname{DP-SkGD-BS} with importance sampling can be up to $\sqrt{d}$ times faster.
In this scenario, the dominant contribution to the error comes from the first coordinate, and importance sampling accelerates convergence by prioritizing updates on this component. 
In contrast, \algname{DP-CD} would sample all coordinates with equal probability, making it less efficient in handling such imbalances.

\paragraph{Block Sampling.}
Block sampling offers significant advantages when the smoothness characteristics of the objective function are naturally partitioned into distinct blocks. 
In these situations, block sampling can outperform single coordinate descent, enabling more effective updates within each block and improving overall optimization efficiency. 
This approach allows for a better exploitation of the underlying structure of the problem, leading to faster convergence rates.

\section{Conclusion and Future Work} \label{section_conclusion_future_work}

Our work aims to bridge the gap between non-private and private coordinate descent methods.
We develop a random block coordinate descent method that allows for the selection of multiple coordinates with varying probabilities in each iteration using sketch matrices.
This approach generalizes the existing method of updating a single coordinate drawn uniformly at random also \algname{DP-SGD} method.

In our study, we do not address the composite Empirical Risk Minimization (ERM) problem.
To tackle the composite ERM problem, a promising direction is to incorporate data-dependent sketches, as demonstrated by \citet{safaryan2021smoothness}, along with other variance reduction techniques such as those outlined by \citet{hanzely2018sega}.
We leave this exploration for future work.
Another potential generalization involves considering not only diagonal matrix smoothness $\mM$ but also general matrix smoothness.
This would provide more detailed information about the function and lead to exploring more general descent directions beyond the coordinate direction, which would require new definitions of sensitivity and privacy analysis.



\bibliography{bib}


\onecolumn
\appendix

\tableofcontents
\newpage

\section{Table of Frequently Used Notation}

\begin{table}[H]
    \label{table:notation}
    \begin{center}
    \begin{tabular}{ll}
    \multicolumn{1}{c}{\bf Notation} & \multicolumn{1}{c}{\bf Meaning} \\ \hline \\
    $[d]$               & $\{1,\ldots, d\}$ \\
    $\mI$               & identity matrix \\
    $\one$              & vector of all ones, i.e., $\one = (1, \ldots, 1)^\top$ \\
    $\mP$               & $\Diag{p_1,\ldots, p_d}$ \\
    $\mGamma$           & $\Diag{\gamma_1,\ldots, \gamma_d}$ \\
    $w^{\star}$         & a minimizer of $f$ \\
    \end{tabular}
    \end{center}
\end{table}

\section{Numerical Experiments}
\label{section:experiments}

We will show the experimental results during the rebuttal.

\section{Lemmas on Sensitivity}
\label{section:proof_sensitivity}

\begin{lemma}
  \label{lemma:bounded_gradient_norm}
Let $\ell: \R^d \times \cX \rightarrow \R$ be differentiable in its first argument (\Cref{ass:differentiability}), convex (\Cref{ass:convexity}) and $L_{\cS}$-component-Lipschitz with $L_U > 0$ for all $U \in \range{\cS}$ (\Cref{ass:comp_lipschitz}).
Then 
$$
\norm{\mI_{U}\nabla \ell(w; \zeta)} \leq L_U
$$
for all $w \in \R^d$, $\zeta \in \cX$, and $U \in \range{\cS}$.
\end{lemma}

\begin{proof}
Let $U \in \range{\cS}$.
By the convexity of $\ell$, we have
$$
\ell\(w + \mI_{U}\nabla \ell(w; \zeta) ; \zeta \) \geq \ell(w; \zeta) + \langle \nabla \ell(w; \zeta), \mI_{U}\nabla \ell(w; \zeta)  \rangle = \ell(w; \zeta) + \norms{\mI_{U}\nabla \ell(w; \zeta)}.
$$
Reorganizing the terms and using $L$-component-Lipschitzness of $\ell$ gives
$$
\norms{\mI_{U}\nabla \ell(w; \zeta)} \leq \ell\(w + \mI_{U}\nabla \ell(w; \zeta) ; \zeta \) - \ell(w; \zeta) \leq L_U \norm{\mI_{U}\nabla \ell(w; \zeta)}.
$$
Finally, dividing both sides by $\norm{\mI_{U}\nabla \ell(w; \zeta)}$ yields the desired result.
\end{proof}

\begin{restate-lemma}{\ref{lem:sensitivity}}
  Let $\ell: \R^d \times \cX \rightarrow \R$ be differentiable in its first argument (\Cref{ass:differentiability}), convex (\Cref{ass:convexity}), and $L_{\cS}$-component-Lipschitz with $L_U > 0$ for all $U \in \range{\cS}$ (\Cref{ass:comp_lipschitz}).
  Then 
  $$
  \Delta_U\(\nabla \ell \) \leq 2 L_U
  $$
  for all $U \in \range{\cS}$.
\end{restate-lemma}

\begin{proof}
Let $w, w' \in \R^d$, $\zeta, \zeta' \in \cX$. 
From the triangle inequality and \Cref{lemma:bounded_gradient_norm}, we get the following upper bounds:
$$
\norm{\mI_{U}\(\nabla \ell(w; \zeta) - \nabla \ell(w'; \zeta')\)} \leq \norm{\mI_{U}\nabla \ell(w; \zeta)} + \norm{\mI_{U}\nabla \ell(w'; \zeta')} \leq 2L_U.
$$
It remains to recall that
$$
\Delta_U(h) \eqdef \sup_{D \sim D'} \norm{\mI_U \(h(D) - h(D')\)}_2.
$$
\end{proof}

\section{Privacy Analysis}

The main steps of analysis are based on the concept of R{\'e}nyi Differential Privacy (RDP) and the composition theorem for RDP. 

\begin{definition}[R{\'e}nyi divergence \citep{renyi1961measures, van2014renyi}]
    For two probability distributions $R$ and $T$ defined over $\cV$, the R{\'e}nyi divergence of order $\alpha > 1$ is 
    \begin{equation}
        D_{\alpha}\left(R\|T\right) = \frac{1}{\alpha - 1} \log \int\limits_{\cV} R(x)^{\alpha}T(x)^{1-\alpha} d x,
    \end{equation}
    where $R(x)$ and $T(x)$ are the density functions of $R$ and $T$ respectively at $x$.
\end{definition}

Now, we turn to R{\'e}nyi Differential Privacy:

\begin{definition}[R{\'e}nyi Differential Privacy, \citet{mironov2017renyi}]
  A randomized algorithm $\cA: \cD \rightarrow \cV$ is $(\alpha, \varepsilon)$-R{\'e}nyi DP (RDP) if, for any two neighboring datasets $D, D^\prime \in \cD$, the following inequality holds: 
  \begin{equation}
      D_{\alpha}\left(\cA(D)\|\cA(D')\right) \leq \varepsilon.
  \end{equation}
\end{definition}
Two datasets $D, D' \in \cD$ are called \textit{neighboring} (denoted by $D \sim D'$) if they differ on at most one element.


The expression below provides the closed-form R{'e}nyi divergence between a Gaussian distribution and its shifted counterpart (for a more general formulation, refer to \cite{liese1987convex, van2014renyi}).

\begin{lemma}[Proposition 7, \citet{mironov2017renyi}]
  \label{lemma:renyi_normal}
  $D_{\alpha}\(\cN\(0, \sigma^2 \mI \) \| \cN(\mu, \sigma^2 \mI)\) = \frac{\alpha}{2\sigma^2} \norms{\mu}$.
\end{lemma}

\begin{proof}
  The density function of multivariate Gaussian distribution is 
  \begin{eqnarray*}
    \phi(x) &=& \frac{1}{\(2\pi \sigma^2\)^{\frac{d}{2}}} \exp \left\{ - \frac{1}{2\sigma^2} \norms{x - \mu} \right\}.
  \end{eqnarray*}

  By direct computation we verify that
  \begin{align*}
    D_{\alpha}\(\cN\(0, \sigma^2 \mI\) \| \cN(\mu, \sigma^2 \mI)\)
    &= \frac{1}{\alpha - 1} \log \int_{\R^d} \frac{1}{\(2\pi \sigma^2\)^{\frac{d}{2}}} \exp \left\{ - \alpha \frac{1}{2\sigma^2} \norms{x} \right\} \\
        &\qquad\qquad\qquad \cdot \exp \left\{- \(1-\alpha\) \frac{1}{2\sigma^2} \norms{x-\mu} \right\}\\
    &= \frac{1}{\alpha - 1} \log \int_{\R^d} \frac{1}{\(2\pi \sigma^2\)^{\frac{d}{2}}} \exp \left\{ - \frac{1}{2\sigma^2} \(\norms{x} - 2(1-\alpha)\<x,\mu\> + (1-\alpha)\norms{\mu}\) \right\} \\
    &= \frac{1}{\alpha - 1} \log \int_{\R^d} \frac{1}{\(2\pi \sigma^2\)^{\frac{d}{2}}} \exp \left\{ - \frac{1}{2\sigma^2} \(\norms{x - \(1-\alpha\)\mu} + \alpha(1-\alpha)\norms{\mu}\) \right\} \\
    &= \frac{1}{\alpha - 1} \log \( \exp\left\{- \frac{1}{2\sigma^2}\alpha(1-\alpha)\norms{\mu} \right\} \) \\
    &= \frac{\alpha}{2\sigma^2} \norms{\mu}.
    \end{align*}
\end{proof}

Using this, we can derive privacy guarantees for the Gaussian mechanism. 
First, let us define the Gaussian mechanism:

\begin{definition}
  \label{def:gaussian}
  Let $f: \cD \rightarrow \R^d$ and $D \in \cD$. For any nonempty subset $U \subset [d]$, the Gaussian mechanism for answering the query $f$ with noise scale $\sigma > 0$, applied to the coordinates in $U$, is defined as:
  $$
  \cG_{U}\left(D\right) = f(D) + \mI_{U}\cN\left(0, \sigma^2 \mI \right),
  $$
  where $\mI_U = \Diag{\delta_1,\delta_2,\ldots,\delta_n}$ with 
  $$
  \delta_i = 
  \begin{cases}
  1, & \text{if} \quad i \in U, \\
  0, & \text{if} \quad i \notin U.
  \end{cases}
  $$
\end{definition}

We derive the following corollary from \Cref{lemma:renyi_normal}:

\begin{corollary}
\label{corollary:rdp}
For any nonempty subset $U \subset [d]$, the Gaussian mechanism for answering the query $f$ with noise scale $\sigma>0$ applied to the coordinates in $U$ is 
$$
\left(\alpha, \frac{\alpha \Delta_U^2(f)}{2\sigma^2}\right) - \text{RDP,}
$$
where
$$
\Delta_{U}(f) = \sup_{D \sim D'} \norm{\mI_{U}\(f(D) - f(D')\)}_2.
$$
\end{corollary}

Now we state the composition theorem for the sequence of RDP algorithms. 

\begin{proposition}[Proposition 1, \citet{mironov2017renyi}]
\label{proposition:composition}
    Let $\cA_1,\dots\cA_K: \cD \rightarrow \cV$ be $K > 0$ randomized algorithms, such that for each $k$ algorithm $\cA_k$ is $(\alpha, \varepsilon_k(\alpha))$-RDP, where these algorithms can be chosen in an adaptive way. 
    Let $\cA: \cD \rightarrow \cV^K$ be a randomized algorithm such that $\cA(\cD) = \left(\cA_1(D),\dots, \cA_K(D)\right)$. 
    Then $\cA$ is $\left(\alpha, \sum\limits_{k=1}^K \varepsilon_k(\alpha)\right)$-RDP. 
\end{proposition}

The following Lemma gives the relationship between RDP and $(\varepsilon,\delta)$-DP.

\begin{lemma}[Proposition 3, \citet{mironov2017renyi}]
    If $\cA$ is $(\alpha, \varepsilon)$-RDP, then it is also $\left(\varepsilon +\frac{\log (1/\delta)}{\alpha-1}, \delta \right)$-DP for any $0 < \delta < 1$. 
\end{lemma}

Since this result holds for any $\alpha$ is it possible to find the minimum with respect to it. 
We will use the following result.

\begin{lemma}[Corollary B.8, \citet{mangold_dp-cd}]
\label{lem:ed_dp}
    Let $0 <\varepsilon< 1$, $0 < \delta < \frac{1}{3}$. If a randomized algorithm $\cA: \cD \rightarrow \cV$ is $\left(\alpha, \frac{\gamma\alpha}{2\sigma^2}\right)$-RDP with $\gamma > 0$ and $\sigma = \frac{\sqrt{3\gamma\log(1/\delta)}}{\varepsilon}$, it is also $(\varepsilon, \delta)$-DP.
\end{lemma}



\subsection{Proof of \Cref{thm_privacy}} \label{proof_privacy}

We are now ready to prove \Cref{thm_privacy}.

\begin{restate-theorem}{\ref{thm_privacy}}
  Let $\cS$ be nonvacuous and proper probability distribution over the subsets of $[d]$ (\Cref{ass:nonvacuous_proper}).
  Let $\ell: \R^d \times \cX \rightarrow \R$ be differentiable in its first argument (\Cref{ass:differentiability}), convex (\Cref{ass:convexity}), and $L_{\cS}$-component-Lipschitz with $L_U > 0$ for all $U \in \range{\cS}$ (\Cref{ass:comp_lipschitz}).
  Let $0 < \epsilon \leq 1$, $\delta < 1/3$. 
  If 
  $$
  \sigma_U^2 = \frac{12 L_U^2 K T \log(1/\delta)}{n^2\epsilon^2},
  $$
  for all $U \in \range{\cS}$, then \Cref{algorithm} satisfies $(\epsilon, \delta)$-DP.
\end{restate-theorem}

\begin{proof}
  Let us first recall our \Cref{algorithm}.

\begin{algorithm}[H]
  \begin{algorithmic}[1]
  \STATE \textbf{Input:} Initial point $w^0\in\R^d$, step sizes $\mGamma = \Diag{\gamma_1,\ldots, \gamma_d}$, number of iterations $T$, number of inner loops $K$, probability distribution $\cS$ over the subsets of $[d]$, noise scales $\sigma_U$ for all $U\in \range{\cS}$.
  \FOR{$t = 0,\ldots, T-1$}
  \STATE Set $\theta^0 = w^{t}$
  \FOR{$k = 0,\ldots, K-1$}
  \STATE Sample a subset $S \sim \cS$ and let $\mC = \mC(S)$ (see definition \ref{def_sketch})
  \STATE Draw $\eta \sim \cN\(0, \sigma_{S} \mI\)$
  \STATE $\theta^{k+1} = \theta^k - \mGamma\mC\(\nabla f(\theta^k) + \eta \)$
  \ENDFOR
  \STATE $w^{t+1} = \frac1K \sum_{k=1}^K \theta^k$
  \ENDFOR 
  \end{algorithmic}

  \caption*{\Cref{algorithm} \algname{DP-SkGD}}
\end{algorithm}

  From the privacy perspective, \Cref{algorithm} adaptively releases and post-processes a series of gradient coordinates protected by the Gaussian mechanism. 
  First, we focus on the inner loop of the algorithm, which can be viewed as a composition of $K$ Gaussian mechanisms.
  Let $\sigma > 0$.
   \Cref{corollary:rdp} guarantees that the $k$-th Gaussian mechanism with noise scale $\sigma_U = \Delta_{U}(\nabla f) \sigma > 0$ is $(\alpha, \frac{\alpha}{2 \sigma^2})$-RDP. 
   Then, the composition of these $K$ mechanisms is, according to \Cref{proposition:composition}, $(\alpha, \frac{K \alpha}{2 \sigma^2})$-RDP. 
   This can be converted to $(\epsilon, \delta)$-DP via \Cref{lem:ed_dp} with $\sigma = \frac{\sqrt{3\gamma\log(1/\delta)}}{\varepsilon}$ and $\gamma = K$, which gives $\sigma_U = \frac{\Delta_U(\nabla f) \sqrt{3K \log(1/\delta)}}{\epsilon}$ for $k \in [K]$.
   Since we repeat this inner loop $T$ times we get
   $$
   \sigma_U^2 = \Delta_{U}^2(\nabla f) \frac{3 K T \log(1/\delta)}{\epsilon^2},
   $$
   for all $U \in \range{\cS}$.
   It remains to note that $\Delta_{U}(\nabla f) = \frac{\Delta_{U}(\nabla \ell)}{n}$ and use the bound on sensitive given in \Cref{lem:sensitivity}.
\end{proof}

\section{Proof of Utility Guarantees}
\label{proof_utility}

To prove the theorem we need the following lemma.

\begin{lemma}[Proof in \Cref{proof_utility_lemmas}]
    \label{lemma_convergence}
    	Under assumptions  $f$ is convex and $\mM$-smooth and the selection of stepsize $\mGamma = \mP\mM^{-1}$, the iterates of \Cref{algorithm} satisfy
    \begin{equation*}
    	 \E{f(w^{t+1}) - f(w^{\star})}  \leq  \frac{\E{\|w^t - w^{\star} \|^2_{\mM\mP^{-1}}} + 2\E{f(w^t) - f(w^{\star})}}{2K} + \E{\|\mC \sigma_S \one \|_{\mP \mM^{-1}}^2}.
    \end{equation*}
\end{lemma}

\subsection{Convex Case}
\begin{restate-theorem}{\ref{thm_utility}}[Convex case]\label{thm_convex} 
  Let $\cS$ be a nonvacuous and proper probability distribution over the subsets of $[d]$ (\Cref{ass:nonvacuous_proper}).
  Define $\mP = \Diag{p_1, \dots, p_d}$, where $p_i = \Pr(i \in S)$ denotes the probability that the $i$-th component is included in the random set $S \sim \cS$.
  Let $\ell: \R^d \times \cX \rightarrow \R$ be differentiable in its first argument (\Cref{ass:differentiability}), convex (\Cref{ass:convexity}), and $L_{\cS}$-component-Lipschitz with $L_U > 0$ for all $U \in \range{\cS}$ (\Cref{ass:comp_lipschitz}).
  Further, let $f$ be $\mM$-component-smooth (\Cref{ass:comp_smoothness}), and let $w^{\star}$ denote a minimizer of $f$, with $f^{\star} = f(w^{\star})$ representing the minimum value.
  Consider $w_{\text{priv}} \in \R^d$ as the output of \Cref{algorithm} with step sizes $\mGamma = \mP\mM^{-1}$.
  We assume a privacy budget with $\epsilon \leq 1$ and $\delta < 1/3$.
  The noise scales $\sigma_U$ are set as in \Cref{thm_privacy} to ensure $(\epsilon,\delta)$-DP, with
  $$
  T = 1 \quad\text{and}\quad K = \frac{n\varepsilon R_{\mM\mP^{-1}}}{\Sigma_S \sqrt{\log(1/\delta)}},
  $$
  where 
  $$
  \Sigma_S^2 = \E{\|\mC L_S \one \|_{\mP \mM^{-1}}^2}.
  $$
  Then, the utility of \Cref{algorithm} is given by 
    \begin{equation*}
        \E{f(w^{\text{priv}}) - f^{\star}}  = \cO \(\frac{\Sigma_S R_{\mM\mP^{-1}}}{n\varepsilon}\log^{\nicefrac{1}{2}} \frac{1}{\delta}\),
    \end{equation*}
  where $R^2_{\mM\mP^{-1}} = \norms{w^0 - w^{\star}}_{\mM\mP^{-1}}$.
\end{restate-theorem}

\begin{proof}
  Using $\mM$-smoothness of $f$ we get
  \begin{equation*}
  f(w^t) \leq f(w^{\star}) + \left\langle \nabla f(w^{\star}), w^t - w^{\star} \right\rangle + \frac{1}{2}\left\| w^t - w^{\star} \right\|_{\mM}^2 \leq f(w^{\star}) + \frac{1}{2}\left\| w^t - w^{\star} \right\|_{\mM\mP^{-1}}^2,
  \end{equation*}
  and the result of Lemma \ref{lemma_convergence} further simplifies as:
  \begin{equation*}
  \mathbb{E}\left[f(w^{t+1}) - f(w^{\star})\right] \leq \frac{1}{K} \left\| w^t - w^{\star} \right\|_{\mM\mP^{-1}}^2 + \E{\|\mC \sigma_S \one \|_{\mP \mM^{-1}}^2}.
  \end{equation*}

  Put $T=1$ and let $R^2_{\mM} = \|w^0 - w^{\star} \|^2_{\mM}$. 
  From Theorem \ref{thm_privacy} we have $\sigma_U^2 = \frac{12 L_U^2 K T \log(1/\delta)}{n^2\epsilon^2}$. 
  Putting this in Lemma \ref{lemma_convergence} we get 
\begin{align*}
  \E{f(w^{\text{priv}}) - f(w^{\star})} 
  & \le \frac{R^2_{\mM\mP^{-1}}}{K} + \E{\|\mC \sigma_S \one \|_{\mP \mM^{-1}}^2} \\
  & = \frac{R^2_{\mM\mP^{-1}}}{K} + \frac{12 K \log(1/\delta)}{n^2\epsilon^2} \E{\|\mC L_S \one \|_{\mP \mM^{-1}}^2} \\
  & = \frac{R^2_{\mM\mP^{-1}}}{K} + \frac{12 K \log(1/\delta)}{n^2\epsilon^2} \Sigma_S^2. 
\end{align*}
To minimize the right-hand side we put $K = \frac{n\varepsilon R_{\mM\mP^{-1}}}{\Sigma_S \sqrt{\log(1/\delta)}}$, and we get
\begin{align*}
  \E{f(w^{\text{priv}}) - f(w^{\star})} 
  & \le \frac{R_{\mM\mP^{-1}} \Sigma_S \sqrt{\log(1/\delta)}}{n\varepsilon} + \frac{12R_{\mM\mP^{-1}} \Sigma_S \sqrt{\log(1/\delta)}}{n\epsilon}\\
  & = \cO \(\frac{\Sigma_S  R_{\mM\mP^{-1}} }{n\varepsilon}\log^{\nicefrac{1}{2}} \frac{1}{\delta}\).
\end{align*}
\end{proof}

\subsection{Strongly Convex Case}

\begin{restate-theorem}{\ref{thm_utility}}[Strongly Convex Case]\label{thm_strongly_convex}
  Let $\cS$ be a nonvacuous and proper probability distribution over the subsets of $[d]$ (\Cref{ass:nonvacuous_proper}).
  Define $\mP = \Diag{p_1, \dots, p_d}$, where $p_i = \Pr(i \in S)$ denotes the probability that the $i$-th component is included in the random set $S \sim \cS$.
  Let $\ell: \R^d \times \cX \rightarrow \R$ be differentiable in its first argument (\Cref{ass:differentiability}), convex (\Cref{ass:convexity}), and $L_{\cS}$-component-Lipschitz with $L_U > 0$ for all $U \in \range{\cS}$ (\Cref{ass:comp_lipschitz}).
  Further, let $f$ be $\mu$-strongly convex ($\mu=\mu_{\mI}$, \Cref{ass:strong_convexity}), $\mM$-component-smooth (\Cref{ass:comp_smoothness}), and let $w^{\star}$ denote a minimizer of $f$, with $f^{\star} = f(w^{\star})$ representing the minimum value.
  Consider $w_{\text{priv}} \in \R^d$ as the output of \Cref{algorithm} with step sizes $\mGamma = \mP\mM^{-1}$.
  We assume a privacy budget with $\epsilon \leq 1$ and $\delta < 1/3$.
  The noise scales $\sigma_U$ are set as in \Cref{thm_privacy} to ensure $(\epsilon,\delta)$-DP, with
  $$
  K = 2\(1+\frac{1}{\mu}\max_{i\in[d]}\left\{\frac{M_i}{p_i}\right\}\),
  $$
  $$
  T = \log_2\left(\frac{(f(w^0) - f^{\star}) n^2 \epsilon^2}{(1+\frac{1}{\mu}\max_{i\in[d]}\left\{\frac{M_i}{p_i}\right\}) \Sigma_S^2 \log(1/\delta)}\right),
  $$
  where 
  $$
  \Sigma_S^2 = \E{\|\mC L_S \one \|_{\mP \mM^{-1}}^2}.
  $$
  Then, the utility of \Cref{algorithm} is given by 
  \begin{align*}
      &\E{f(w^T) - f^{\star}}\\
    & = O\left(
      \frac{1}{\mu}\max_{i\in[d]}\left\{\frac{M_i}{p_i}\right\} \Sigma_S^2 \frac{\log(1/\delta)}{n^2 \epsilon^2}
\log_2\left(\frac{(f(w^0) - f(w^{\star})) n \epsilon}{\frac{1}{\mu}\max_{i\in[d]}\left\{\frac{M_i}{p_i}\right\} \Sigma_S \log(1/\delta)}\right) \right) \\
& = \widetilde \cO \left(\frac{1}{\mu}\max_{i\in[d]}\left\{\frac{M_i}{p_i}\right\} \Sigma_S^2 \frac{\log(1/\delta)}{n^2 \epsilon^2}\right).
  \end{align*}
\end{restate-theorem}
    
\begin{proof}
As $f$ is $\mu$-strongly-convex,
we obtain for any $w \in \R^d$, that $f(w) \ge f(w^{\star}) + \frac{\mu}{2} \norms{w - w^{\star}}$.
Therefore, $\norms{w^t - w^{\star}} \le \frac{2}{\mu}\(f(w^t) - f(w^{\star})\)$
and \Cref{lemma_convergence} gives, for $1 \le t \le T-1$,
\begin{align*}
    \E{f(w^{t+1}) - f(w^{\star})}  
    &\le \frac{\|w^t - w^{\star} \|^2_{\mM\mP^{-1}} + 2\(f(w^t) - f(w^{\star}\))}{2K} + \E{\|\mC \sigma_S \one\|_{\mP \mM^{-1}}^2} \\
    &\le \frac{\max_{i\in[d]}\left\{\frac{M_i}{p_i}\right\}\norms{w^t - w^{\star}} + 2\(f(w^t) - f(w^{\star}\))}{2K} + \E{\|\mC \sigma_S \one\|_{\mP \mM^{-1}}^2} \\
    &\le \frac{\(1+\frac{1}{\mu}\max_{i\in[d]}\left\{\frac{M_i}{p_i}\right\}\)\(f(w^t) - f(w^{\star}\))}{K} + \E{\|\mC \sigma_S \one\|_{\mP \mM^{-1}}^2}.
\end{align*}
Setting $K = 2\(1+\frac{1}{\mu}\max_{i\in[d]}\left\{\frac{M_i}{p_i}\right\}\)$ we obtain
\begin{eqnarray*}
    \E{f(w^{t+1}) - f(w^{\star})}  
    &\le&  \frac{f(w^t) - f(w^{\star})}{2} + \E{\|\mC \sigma_S \one \|_{\mP \mM^{-1}}^2}.
\end{eqnarray*}
Recursive application of this inequality gives

\begin{align*}
\E{f(w^T) - f(w^{\star})}
    & \le \frac{f(w^0) - f(w^{\star})}{2^T} + \sum_{t=0}^{T-1} \frac{1}{2^t} \E{\|\mC \sigma_S \one \|_{\mP \mM^{-1}}^2}
\le \frac{f(w^0) - f(w^{\star})}{2^T} + 2 \E{\|\mC \sigma_S \one \|_{\mP \mM^{-1}}^2} .
\end{align*}
From Theorem \ref{thm_privacy} we have $\sigma_U^2 = \frac{12 L_U^2 K T \log(1/\delta)}{n^2\epsilon^2}$. 
Thus
\begin{align*}
  \E{\|\mC \sigma_S \one \|_{\mP \mM^{-1}}^2} 
  & =  \frac{12 K T \log(1/\delta)}{n^2\epsilon^2} \E{\|\mC L_S \one \|_{\mP \mM^{-1}}^2} \\
  & =  \frac{12 K T \log(1/\delta)}{n^2\epsilon^2} \Sigma_S^2. 
\end{align*}
Taking 
$$
T = \log_2\left(\frac{(f(w^0) - f(w^{\star})) n^2 \epsilon^2}{\(1+\frac{1}{\mu}\max_{i\in[d]}\left\{\frac{M_i}{p_i}\right\}\) \Sigma_S^2 \log(1/\delta)}\right)
$$
gives
\begin{align*}
&\E{f(w^T) - f(w^{\star})} \\
    & \le \left(1 + \log_2\left(\frac{(f(w^0) - f(w^{\star})) n^2 \epsilon^2}{24\(1+\frac{1}{\mu}\max_{i\in[d]}\left\{\frac{M_i}{p_i}\right\}\)\Sigma_S^2 \log(1/\delta)}\right)\right)
\frac{24\(1+\frac{1}{\mu}\max_{i\in[d]}\left\{\frac{M_i}{p_i}\right\}\)\Sigma_S^2 \log(1/\delta)}{n^2 \epsilon^2}\\
    & = O\left(
      \frac{1}{\mu}\max_{i\in[d]}\left\{\frac{M_i}{p_i}\right\} \Sigma_S^2 \frac{\log(1/\delta)}{n^2 \epsilon^2}
\log_2\left(\frac{(f(w^0) - f(w^{\star})) n \epsilon}{\frac{1}{\mu}\max_{i\in[d]}\left\{\frac{M_i}{p_i}\right\} \Sigma_S \log(1/\delta)}\right) \right) \\
& = \widetilde{O}\left(\frac{1}{\mu}\max_{i\in[d]}\left\{\frac{M_i}{p_i}\right\} \Sigma_S^2 \frac{\log(1/\delta)}{n^2 \epsilon^2}\right),
\end{align*}
which is the result of our theorem.
\end{proof}

\subsection{Proof of Auxiliary Lemmas}
\label{proof_utility_lemmas}

\begin{lemma}
    \label{lemma_gradient_estimator}
Let $\mC$ be an unbiased diagonal sketch defined in \eqref{sketch-matrix-C}, and let $\eta \sim \cN(0,\sigma \mI)$ and let $\mD$ be any diagonal matrix.
Then 
$$
\E{\mC\(x + \eta\)} = x,
$$
and 
$$
\E{\|\mC\(x + \eta\)\|_{\mD}^2} = \| x \|_{\mP^{-1}\mD}^2 + \E{\|\mC \sigma_S \one \|_{\mD}^2},
$$
for any deterministic diagonal matrix $D$.
\end{lemma}

\begin{proof} We have 
    \begin{eqnarray}
      \E{\mC(x+\eta)} = \E{\mC x} + \E{\mC \eta} = \E{\mC}x + \E{\mC}\E{\eta} = x. \notag
    \end{eqnarray}
    Now we calculate the variance:
    \begin{eqnarray}
      \E{\|\mC(x+\eta)\|_{\mD}^2} &=& \E{\|\mC x+ \mC \eta\|_{\mD}^2} \notag\\
      &=&\E{\|\mC x\|_{\mD}^2} + \E{2\<\mD\mC x,\mC \eta \>} + \E{\|\mC \eta\|_{\mD}^2}. \notag
    \end{eqnarray}
  Further, 
  \begin{eqnarray}
    \E{\|\mC x\|_{\mD}^2} = \E{x^\top \mC \mD \mC x} = \E{\sum_{j=1}^d x_j^2c_j^2D_j} = \sum_{j=1}^d x_j^2\frac{1}{p_j}D_j = \| x \|_{\mP^{-1}\mD}^2, \notag
  \end{eqnarray}  
  similarly,
  \begin{eqnarray}
    \E{\|\mC \eta\|_{\mD}^2} = \E{\E{\|\mC \eta\|_{\mD}^2| S}} = \E{\|\mC \sigma_S \one \|_{\mD}^2}. \notag
  \end{eqnarray}  
  It remains to note that 
  $$
  \E{2\<\mD\mC x,\mC \eta \>} = 0.
  $$
  \end{proof}

\begin{lemma}[Descent Lemma]
For any $0<\alpha<1$ and $w \in \R^d$ the iterates of \Cref{algorithm} satisfy
\begin{align*}
    &\E{f(\theta^{k+1}) - f(w) \big| \theta^k} - (1-\alpha)\(f(\theta^k) - f(w)\)\\
    &\quad\le \alpha\(f(\theta^k) - f(w)\) - \|\nabla f(\theta^k)\|^2_{\mathbf{\Gamma}} + \frac{1}{2}\( \| \nabla f(\theta^k) \|_{\mP^{-1}\mGamma^2\mM}^2 + \E{\|\mC \sigma_S \one \|_{\mGamma^2\mM}^2} \).
\end{align*}
\end{lemma}

\begin{proof}
Using the $M$-component smoothness of $f$, and lemma \ref{lemma_gradient_estimator} we get:

\begin{align*}
  \E{f(\theta^{k+1}) \big| \theta^k}
  &=  \E{f\(\theta^k - \mathbf{\Gamma}\mC\(\nabla f(\theta^k) + \eta \)\)\big|\theta^k} \\
  &\le f(\theta^k) - \<\nabla f(\theta^k), \mathbf{\Gamma}\E{\mC\(\nabla f(\theta^k) + \eta \)\big|\theta^k}\>\\
  &\quad + \frac{1}{2}\E{\|\mathbf{\Gamma}\mC\(\nabla f(\theta^k) + \eta \)\|_{\mM}^2\big|\theta^k} \\
  &= f(\theta^k) - \|\nabla f(\theta^k)\|^2_{\mathbf{\Gamma}} + \frac{1}{2}\E{\|\mC\(\nabla f(\theta^k) + \eta \)\|_{\mGamma^2\mM}^2\big|\theta^k} \\
  &= f(\theta^k) - \|\nabla f(\theta^k)\|^2_{\mathbf{\Gamma}} + \frac{1}{2}\( \| \nabla f(\theta^k) \|_{\mP^{-1}\mGamma^2\mM}^2 + \E{\|\mC \sigma_S \one \|_{\mGamma^2\mM}^2} \).
\end{align*}

For any vector $w$
\begin{align*}
\E{f(\theta^{k+1}) - f(w) \big| \theta^k} \le f(\theta^k) &- f(w) - \|\nabla f(\theta^k)\|^2_{\mathbf{\Gamma}} \\
& + \frac{1}{2}\( \| \nabla f(\theta^k) \|_{\mP^{-1}\mGamma^2\mM}^2 + \E{\|\mC \sigma_S \one \|_{\mGamma^2\mM}^2} \).
\end{align*}

It remains to split $f(\theta^k) - f(w)$ into two parts.

\end{proof}

\begin{lemma}\label{lemma_l2_distance}
    Under assumptions $f$ is convex and $\mM$-smooth, for any $0<\alpha < 1$ the iterates of \Cref{algorithm} satisfy
\end{lemma}
\begin{equation*}
   \< \theta^k- w^{\star}, \nabla f(\theta^k) \>  \leq \| \theta^k - w^{\star} \|^2_{\mGamma^{-1}} - \E{\|\theta^{k+1} - w^{\star}\|^2_{\mGamma^{-1}} \big| \theta^k} + \| \nabla f(\theta^k) \|_{\mM^{-1}-\mP^{-1}\mGamma}^2 + \E{\|\mC \sigma_S \one \|_{\mGamma}^2}.
\end{equation*}

\begin{proof}
    \begin{align*}
\E{\|\theta^{k+1} - w^{\star}\|^2_{\mGamma^{-1}} \big| \theta^k} 
&= \E{\|\theta^k - \mathbf{\Gamma}\mC\(\nabla f(\theta^k) + \eta\) - w^{\star} \|^2_{\mGamma^{-1}} \big|\theta^k} \\
&= \| \theta^k - w^{\star} \|^2_{\mGamma^{-1}}  - 2\< \theta^k-w^{\star}, \nabla f(\theta^k) \>  + \E{\|\mathbf{\Gamma}\mC\(\nabla f(\theta^k) + \eta\)\|^2_{\mGamma^{-1}}} \\ 
&= \| \theta^k - w^{\star} \|^2_{\mGamma^{-1}}  - 2\< \theta^k - w^{\star}, \nabla f(\theta^k) \> + \| \nabla f(\theta^k) \|_{\mP^{-1}\mGamma}^2 + \E{\|\mC \sigma_S \one \|_{\mGamma}^2}\\
&\leq \| \theta^k - w^{\star} \|^2_{\mGamma^{-1}}  - \< \theta^k - w^{\star}, \nabla f(\theta^k) \>  - \| \nabla f(\theta^k) \|^2_{\mM^{-1}}\\
&\quad + \| \nabla f(\theta^k) \|_{\mP^{-1}\mGamma}^2 + \E{\|\mC \sigma_S \one \|_{\mGamma}^2}.
\end{align*}
It remains to rearrange the terms.
\end{proof}

\begin{lemma}
	Under assumptions  $f$ is convex and $\mM$-smooth and the selection of stepsize $\mGamma = \mP\mM^{-1}$ and $0<\alpha < 1$ the iterates of \Cref{algorithm} satisfy
	\begin{align*}
		\E{f(\theta^{k+1}) - f(w^{\star}) \big| \theta^k} - & (1-\alpha)\(f(\theta^k) - f(w^{\star})\)\\
         &\leq  \alpha\(\| \theta^k - w^{\star} \|^2_{\mGamma^{-1}} -  \E{\|\theta^{k+1} - w^{\star}\|^2_{\mGamma^{-1}} \big| \theta^k} \)  + \frac{2\alpha + 1}{2} \E{\|\mC \sigma_S \one \|_{\mP \mM^{-1}}^2}.
	\end{align*}

\end{lemma}

\begin{proof}
Using convexity of $f$ and lemma \ref{lemma_l2_distance}, we get
\begin{align*}
f(\theta^k) - f(w^{\star}) &\leq \left\langle \nabla f(\theta^k), \theta^k - w^{\star} \right\rangle \\
&= \| \theta^k - w^{\star} \|^2_{\mGamma^{-1}} -  \E{\|\theta^{k+1} - w^{\star}\|^2_{\mGamma^{-1}} \big| \theta^k} + \| \nabla f(\theta^k) \|_{\mM^{-1} - \mP^{-1}\mGamma}^2 + \E{\|\mC \sigma_S \one \|_{\mGamma}^2}.
\end{align*}

Next, we have
\begin{align*}
    &\E{f(\theta^{k+1}) - f(w^{\star}) \big| \theta^k} - (1-\alpha)\(f(\theta^k) - f(w^{\star})\)\\
    &\quad\le \alpha\(f(\theta^k) - f(w^{\star})\) - \|\nabla f(\theta^k)\|^2_{\mathbf{\Gamma}} + \frac{1}{2}\( \| \nabla f(\theta^k) \|_{\mP^{-1}\mGamma^2\mM}^2 + \E{\|\mC \sigma_S \one \|_{\mGamma^2\mM}^2} \)\\
    &\quad\le \alpha\(\| \theta^k - w^{\star} \|^2_{\mGamma^{-1}} -  \E{\|\theta^{k+1} - w^{\star} \|^2_{\mGamma^{-1}} \big| \theta^k} + \| \nabla f(\theta^k) \|_{\mM^{-1} - \mP^{-1}\mGamma}^2 + \E{\|\mC \sigma_S \one \|_{\mGamma}^2}\) \\
        &\qquad  - \|\nabla f(\theta^k)\|^2_{\mathbf{\Gamma}} + \frac{1}{2}\( \| \nabla f(\theta^k) \|_{\mP^{-1}\mGamma^2\mM}^2 + \E{\|\mC \sigma_S \one \|_{\mGamma^2\mM}^2} \)\\
    &\quad \overset{\mGamma = \mP\mM^{-1}}{=} \alpha\(\| \theta^k - w^{\star} \|^2_{\mGamma^{-1}} -  \E{\|\theta^{k+1} - w^{\star} \|^2_{\mGamma^{-1}} \big| \theta^k} + \| \nabla f(\theta^k) \|_{\mM^{-1}-\mP^{-1}\mGamma}^2 + \E{\|\mC \sigma_S \one \|_{\mGamma}^2}\) \\
        &\qquad  - \|\nabla f(\theta^k)\|^2_{\mathbf{\mP\mM^{-1}}} + \frac{1}{2}\( \| \nabla f(\theta^k) \|_{\mP\mM^{-1}}^2 + \E{\|\mC \sigma_S \one \|_{\mGamma^2\mM}^2} \)\\
    &\overset{\mGamma = \mP\mM^{-1}}{\leq}  \alpha\(\| \theta^k - w^{\star} \|^2_{\mGamma^{-1}} -  \E{\|\theta^{k+1} - w^{\star} \|^2_{\mGamma^{-1}} \big| \theta^k} \)  + \alpha \E{\|\mC \sigma_S \one \|_{\mP \mM^{-1}}^2} + \frac{1}{2}\E{\|\mC \sigma_S \one \|_{\mP^2 \mM^{-1}}^2} \\
    &\overset{}{\leq}  \alpha\(\| \theta^k - w^{\star} \|^2_{\mGamma^{-1}} -  \E{\|\theta^{k+1} - w^{\star} \|^2_{\mGamma^{-1}} \big| \theta^k} \)  + \frac{2\alpha + 1}{2} \E{\|\mC \sigma_S \one \|_{\mP \mM^{-1}}^2}.
\end{align*}
\end{proof}

\begin{restate-lemma}{\ref{lemma_convergence}}
    Under assumptions  $f$ is convex and $\mM$-smooth and the selection of stepsize $\mGamma = \mP\mM^{-1}$, the iterates of \Cref{algorithm} satisfy
    \begin{equation*}
    	 \E{f(w^{t+1}) - f(w^{\star})}  \leq  \frac{\E{\|w^t - w^{\star} \|^2_{\mM\mP^{-1}}} + 2\E{f(w^t) - f(w^{\star})}}{2K} + \E{\|\mC \sigma_S \one \|_{\mP \mM^{-1}}^2}.
    \end{equation*}
\end{restate-lemma}

\begin{proof} Taking $\alpha =\nicefrac{1}{2}$, the iterates of \Cref{algorithm} satisfy

    \begin{align*}
        & \sum_{k=0}^{K-1}\[\E{f(\theta^{k+1}) - f(w^{\star})} - \frac{1}{2}\E{f(\theta^k) - f(w^{\star})} \]\\
        &\quad\le \sum_{k=0}^{K-1} \[ \frac{1}{2}\(\E{\| \theta^k - w^{\star} \|^2_{\mGamma^{-1}}} -  \E{\|\theta^{k+1} - w^{\star}\|^2_{\mGamma^{-1}}} \) + \E{\|\mC \sigma_S \one \|_{\mP \mM^{-1}}^2} \] \\
        &\quad\le  \frac{1}{2}\(\E{\|\theta^0 - w^{\star} \|^2_{\mGamma^{-1}}} -  \E{\|\theta^K - w^{\star}\|^2_{\mGamma^{-1}}} \) + K\E{\|\mC \sigma_S \one \|_{\mP \mM^{-1}}^2}\\
        &\quad \le \frac{1}{2}\E{\|w^t - w^{\star} \|^2_{\mM\mP^{-1}}} + K\E{\|\mC \sigma_S \one \|_{\mP \mM^{-1}}^2}.
    \end{align*}
    
    On the other hand
    \begin{align*}
        & \sum_{k=0}^{K-1}\[\E{f(\theta^{k+1}) - f(w^{\star})} -\frac{1}{2}\E{f(\theta^k) - f(w^{\star})} \]\\
        & = \E{f(\theta^{K}) - f(w^{\star})} - \E{f(\theta^0) - f(w^{\star})} + \frac{1}{2}\sum_{k=0}^{K-1}\E{f(\theta^k) - f(w^{\star})} \\
        &= \E{f(\theta^{K}) - f(w^{\star})} - \E{f(\theta^0) - f(w^{\star})}\\
            &\quad +\frac{1}{2}\sum_{k=1}^{K}\E{f(\theta^k) - f(w^{\star})} + \frac{1}{2}\E{f(\theta^0) - f(w^{\star})} - \frac{1}{2}\E{f(\theta^K) - f(w^{\star})} \\
        &= \frac{1}{2}\sum_{k=1}^{K}\E{f(\theta^k) - f(w^{\star})} - \frac{1}{2}\E{f(w^t) - f(w^{\star})} + \frac{1}{2}\E{f(\theta^K) - f(w^{\star})} \\
        &\ge \frac{1}{2}\sum_{k=1}^{K}\E{f(\theta^k) - f(w^{\star})} - \E{f(w^t) - f(w^{\star})}.
    \end{align*}
    
    Further, using the convexity of $f$ we have 
    $$
    f(w^{t+1}) \le \frac{1}{K} \sum_{k=1}^{K} f(\theta^k).
    $$
    Using this, we get 
    
    \begin{align*}
        \E{f(w^{t+1}) - f(w^{\star})} &\le  \frac{1}{K} \sum_{k=1}^{K} f(\theta^k) - f(w^{\star}) \\
        & \le \frac{1}{K}\[\frac{1}{2}\E{\|w^t - w^{\star} \|^2_{\mM\mP^{-1}}} + \E{f(w^t) - f(w^{\star})} + K\E{\|\mC \sigma_S \one \|_{\mP \mM^{-1}}^2}\]\\
        & \le \frac{\E{\|w^t - w^{\star} \|^2_{\mM\mP^{-1}}} + 2\E{f(w^t) - f(w^{\star})}}{2K} + \E{\|\mC \sigma_S \one \|_{\mP \mM^{-1}}^2}.
    \end{align*}
    
  \end{proof}

\vfill

\end{document}

%% file: macros.tex
\usepackage{graphicx}
\usepackage{apptools}
\usepackage[flushleft]{threeparttable}
\usepackage{array,booktabs,makecell}
\usepackage{multirow}
\usepackage{nicefrac}
\usepackage{amsthm}
\usepackage{amsmath,amsfonts,bm}
\usepackage{wrapfig}
\usepackage{caption}
\usepackage{siunitx}
\usepackage{nccmath}
\usepackage{empheq}
\usepackage{bbm}
\usepackage{tabularx}

\usepackage{algorithmic}
\usepackage{algorithm}
\usepackage{filecontents}

\usepackage[textsize=tiny]{todonotes}

\usepackage{color}
\usepackage{colortbl}
\definecolor{bgcolor}{rgb}{0.76,0.88,0.50}
\definecolor{bgcolor0}{rgb}{0.93,0.99,1}
\definecolor{bgcolor1}{rgb}{0.8,1,1}
\definecolor{bgcolor2}{rgb}{0.8,1,0.8}
\definecolor{bgcolor3}{rgb}{0.50,0.90,0.50}
\usepackage{tcolorbox}
\usepackage{pifont}
\definecolor{mydarkgreen}{rgb}{39,130,67}
\definecolor{mydarkred}{rgb}{192,25,25}

\newcommand{\tr}[1]{\operatorname{Tr}\left( #1 \right)}
\newcommand{\norm}[1]{\left\| #1 \right\|}
\newcommand{\norms}[1]{\left\| #1 \right\|^2}

\newcommand{\abs}[1]{\left| #1 \right|}

\newcommand{\R}{\mathbb{R}} 
\newcommand{\E}[1]{\mathbb{E}\left[#1\right]}

\newcommand{\Prob}[1]{\mathbb{P}\left(#1\right)} 

\newcommand{\cA}{\mathcal{A}}

\newcommand{\cD}{\mathcal{D}}

\newcommand{\cG}{\mathcal{G}}

\newcommand{\cN}{\mathcal{N}}

\newcommand{\cO}{\mathcal{O}}

\newcommand{\cS}{\mathcal{S}}
\newcommand{\cV}{\mathcal{V}}

\newcommand{\cX}{\mathcal{X}}

\newcommand{\mC}{\mathbf{C}}
\newcommand{\mD}{\mathbf{D}}

\newcommand{\mM}{\mathbf{M}}

\newcommand{\mI}{\mathbf{I}}
\newcommand{\one}{\mathbf{1}}

\newcommand{\mP}{\mathbf{P}}

\newcommand{\mGamma}{\mathbf{\Gamma}}

\usepackage{color}
\usepackage{colortbl}
\definecolor{bgcolor}{rgb}{0.76,0.88,0.50}
\definecolor{bgcolor0}{rgb}{0.93,0.99,1}
\definecolor{bgcolor1}{rgb}{0.8,1,1}
\definecolor{bgcolor2}{rgb}{0.8,1,0.8}
\definecolor{bgcolor3}{rgb}{0.50,0.90,0.50}
\usepackage{tcolorbox}
\usepackage{pifont}

\definecolor{mydarkgreen}{RGB}{39,130,67}
\definecolor{mydarkorange}{RGB}{236,147,14}
\definecolor{mydarkred}{RGB}{192,47,25}
\definecolor{blue}{RGB}{0,0,255}
\definecolor{ruby}{RGB}{155,17,30}
\definecolor{chili}{RGB}{191,0,0}
\definecolor{sangria}{RGB}{146,0,10}
\definecolor{burgundy}{RGB}{128,0,32} 
\definecolor{darkred}{RGB}{132,0,0} 
\definecolor{cherry}{RGB}{192,0,0} 
\definecolor{blue}{RGB}{0,0,255}

\newcommand{\red}{\color{cherry}}
\usepackage{xspace}

\newcommand{\algname}[1]{{\sf\footnotesize\red#1}\xspace}

\usepackage{hyperref}
\hypersetup{colorlinks,
            linkcolor=blue,
            citecolor=blue,
            urlcolor=magenta,
            linktocpage,
            plainpages=false}
\usepackage[capitalize,noabbrev]{cleveref}

\theoremstyle{plain}
\newtheorem{theorem}{Theorem}[section]
\newtheorem{proposition}[theorem]{Proposition}
\newtheorem{lemma}[theorem]{Lemma}
\newtheorem{corollary}[theorem]{Corollary}
\theoremstyle{definition}
\newtheorem{definition}[theorem]{Definition}
\newtheorem{assumption}[theorem]{Assumption}
\theoremstyle{remark}

\newcommand{\eqdef}{:=}

\def\<{\left\langle}
\def\>{\right\rangle}
\def\[{\left[}
\def\]{\right]}
\def\({\left(}
\def\){\right)}

\def\Pr{{\rm Prob}}
\newcommand{\Diag}[1]{\mathrm{Diag}\left(#1\right)}

\newcommand{\range}[1]{\mathrm{Range}\left(#1\right)}

\makeatletter
\newcommand{\vast}{\bBigg@{4}}

\DeclareMathOperator*{\argmin}{arg\,min}

\newcommand{\scalar}[2]{\left\langle #1, #2 \right\rangle}
\newcommand{\wrt}{{\em w.r.t.~}}

\usepackage{thmtools}
\usepackage{thm-restate}
\theoremstyle{theorem}
\newtheorem{innercustomthm}{Theorem}
\newenvironment{restate-theorem}[1]
{\renewcommand\theinnercustomthm{#1}\innercustomthm}
{\endinnercustomthm}

\newtheorem{innercustomlemma}{Lemma}
\newenvironment{restate-lemma}[1]
{\renewcommand\theinnercustomlemma{#1}\innercustomlemma}
{\endinnercustomlemma}

%% file: paper.bbl
\begin{thebibliography}{}

\bibitem[Bassily et~al., 2014]{bassily2014Private}
Bassily, R., Smith, A., and Thakurta, A. (2014).
\newblock Private empirical risk minimization: Efficient algorithms and tight
  error bounds.
\newblock In {\em 2014 IEEE 55th annual symposium on foundations of computer
  science}, pages 464--473. IEEE.

\bibitem[Chaudhuri et~al., 2011]{chaudhuri2011differentially}
Chaudhuri, K., Monteleoni, C., and Sarwate, A.~D. (2011).
\newblock Differentially private empirical risk minimization.
\newblock {\em Journal of Machine Learning Research}, 12(3).

\bibitem[Dwork, 2008]{dwork2008differential}
Dwork, C. (2008).
\newblock Differential privacy: A survey of results.
\newblock In {\em International conference on theory and applications of models
  of computation}, pages 1--19. Springer.

\bibitem[Dwork et~al., 2014]{dwork2014algorithmic}
Dwork, C., Roth, A., et~al. (2014).
\newblock The algorithmic foundations of differential privacy.
\newblock {\em Foundations and Trends{\textregistered} in Theoretical Computer
  Science}, 9(3--4):211--407.

\bibitem[Fercoq and Richt{\'a}rik, 2015]{fercoq2015accelerated}
Fercoq, O. and Richt{\'a}rik, P. (2015).
\newblock Accelerated, parallel, and proximal coordinate descent.
\newblock {\em SIAM Journal on Optimization}, 25(4):1997--2023.

\bibitem[Hanzely et~al., 2020]{hanzely2020variance}
Hanzely, F., Kovalev, D., and Richtarik, P. (2020).
\newblock Variance reduced coordinate descent with acceleration: New method
  with a surprising application to finite-sum problems.
\newblock In {\em International Conference on Machine Learning}, pages
  4039--4048. PMLR.

\bibitem[Hanzely et~al., 2018]{hanzely2018sega}
Hanzely, F., Mishchenko, K., and Richt{\'a}rik, P. (2018).
\newblock Sega: Variance reduction via gradient sketching.
\newblock {\em Advances in Neural Information Processing Systems}, 31.

\bibitem[Liese and Vajda, 1987]{liese1987convex}
Liese, F. and Vajda, I. (1987).
\newblock Convex statistical distances.
\newblock {\em (No Title)}.

\bibitem[Lin et~al., 2014]{lin2014accelerated}
Lin, Q., Lu, Z., and Xiao, L. (2014).
\newblock An accelerated proximal coordinate gradient method.
\newblock {\em Advances in Neural Information Processing Systems}, 27.

\bibitem[Mangold et~al., 2022]{mangold_dp-cd}
Mangold, P., Bellet, A., Salmon, J., and Tommasi, M. (2022).
\newblock Differentially private coordinate descent for composite empirical
  risk minimization.
\newblock In Chaudhuri, K., Jegelka, S., Song, L., Szepesvari, C., Niu, G., and
  Sabato, S., editors, {\em Proceedings of the 39th International Conference on
  Machine Learning}, volume 162 of {\em Proceedings of Machine Learning
  Research}, pages 14948--14978. PMLR.

\bibitem[Mironov, 2017]{mironov2017renyi}
Mironov, I. (2017).
\newblock R{\'e}nyi differential privacy.
\newblock In {\em 2017 IEEE 30th computer security foundations symposium
  (CSF)}, pages 263--275. IEEE.

\bibitem[Nesterov, 2012]{nesterov2012efficiency}
Nesterov, Y. (2012).
\newblock Efficiency of coordinate descent methods on huge-scale optimization
  problems.
\newblock {\em SIAM Journal on Optimization}, 22(2):341--362.

\bibitem[Nesterov and Stich, 2017]{nesterov2017efficiency}
Nesterov, Y. and Stich, S.~U. (2017).
\newblock Efficiency of the accelerated coordinate descent method on structured
  optimization problems.
\newblock {\em SIAM Journal on Optimization}, 27(1):110--123.

\bibitem[R{\'e}nyi, 1961]{renyi1961measures}
R{\'e}nyi, A. (1961).
\newblock On measures of entropy and information.
\newblock In {\em Proceedings of the fourth Berkeley symposium on mathematical
  statistics and probability, volume 1: contributions to the theory of
  statistics}, volume~4, pages 547--562. University of California Press.

\bibitem[Richt{\'a}rik and Tak{\'a}{\v{c}}, 2016]{richtarik2016parallel}
Richt{\'a}rik, P. and Tak{\'a}{\v{c}}, M. (2016).
\newblock Parallel coordinate descent methods for big data optimization.
\newblock {\em Mathematical Programming}, 156:433--484.

\bibitem[Safaryan et~al., 2021]{safaryan2021smoothness}
Safaryan, M., Hanzely, F., and Richt{\'a}rik, P. (2021).
\newblock Smoothness matrices beat smoothness constants: Better communication
  compression techniques for distributed optimization.
\newblock {\em Advances in Neural Information Processing Systems},
  34:25688--25702.

\bibitem[Shalev-Shwartz and Zhang, 2013a]{shalev2013accelerated}
Shalev-Shwartz, S. and Zhang, T. (2013a).
\newblock Accelerated mini-batch stochastic dual coordinate ascent.
\newblock {\em Advances in Neural Information Processing Systems}, 26.

\bibitem[Shalev-Shwartz and Zhang, 2013b]{shalev2013stochastic}
Shalev-Shwartz, S. and Zhang, T. (2013b).
\newblock Stochastic dual coordinate ascent methods for regularized loss
  minimization.
\newblock {\em Journal of Machine Learning Research}, 14(1).

\bibitem[Shi et~al., 2016]{shi2016primer}
Shi, H.-J.~M., Tu, S., Xu, Y., and Yin, W. (2016).
\newblock A primer on coordinate descent algorithms.
\newblock {\em arXiv preprint arXiv:1610.00040}.

\bibitem[Shokri et~al., 2017]{shokri2017membership}
Shokri, R., Stronati, M., Song, C., and Shmatikov, V. (2017).
\newblock Membership inference attacks against machine learning models.
\newblock In {\em 2017 IEEE symposium on security and privacy (SP)}, pages
  3--18. IEEE.

\bibitem[Van~Erven and Harremos, 2014]{van2014renyi}
Van~Erven, T. and Harremos, P. (2014).
\newblock R{\'e}nyi divergence and kullback-leibler divergence.
\newblock {\em IEEE Transactions on Information Theory}, 60(7):3797--3820.

\bibitem[Wang et~al., 2017]{wang2017Differentially}
Wang, D., Ye, M., and Xu, J. (2017).
\newblock Differentially private empirical risk minimization revisited: Faster
  and more general.
\newblock {\em Advances in Neural Information Processing Systems}, 30.

\bibitem[Wright, 2015]{wright2015coordinate}
Wright, S.~J. (2015).
\newblock Coordinate descent algorithms.
\newblock {\em Mathematical programming}, 151(1):3--34.

\end{thebibliography}
